\documentclass[10pt,reqno]{article}

\usepackage{amssymb,amsmath,comment}
\usepackage[draft]{graphicx}
\usepackage{amsthm}
\usepackage{tikz}
\usepackage{multicol}
\usepackage[colorlinks=true,
linkcolor=webgreen,
filecolor=webbrown,
citecolor=webgreen]{hyperref}
\usepackage[makeroom]{cancel}
\usepackage{ytableau}
\definecolor{webgreen}{rgb}{0,.5,0}
\definecolor{webbrown}{rgb}{.6,0,0}

\usepackage{mathtools}

\theoremstyle{plain}
\newtheorem{theorem}{Theorem}

\newtheorem{proposition}[theorem]{Proposition}

\theoremstyle{definition}
\newtheorem{definition}[theorem]{Definition}
\newtheorem{example}[theorem]{Example}

\theoremstyle{remark}
\newtheorem{remark}[theorem]{Remark}

\def\0g{{\bold 0}}

\def\Z{\Bbb Z}

\def\s{\smallskip}
\def\m{\medskip}

\def\noi{\noindent}

\let\emptyset\varnothing

\title{\Large Some partition and analytical identities arising from Alladi, Andrews, Gordon bijections}

\date{}

\author{S.\, Capparelli\thanks{Dipartimento di Scienze di Base e Applicate per l'Ingegneria, Sapienza Universit\`a di Roma, via Scarpa 16, 00161 Rome (Italy); stefano.capparelli@uniroma1.it; 0000-0002-7924-1688.}, A.\, Del Fra\thanks{Rome, Italy.}, P.\, Mercuri\thanks{Istituto Nazionale di Alta Matematica ``F. Severi'', Rome (Italy); mercuri.ptr@gmail.com; 0000-0003-4402-6432.}, A.\,Vietri\thanks{Dipartimento di Scienze di Base e Applicate per l'Ingegneria, Sapienza Universit\`a di Roma, via Scarpa 16, 00161 Rome (Italy); andrea.vietri@uniroma1.it; 0000-0002-6064-7987.}}

\begin{document}

\maketitle


\vskip1cm

\begin{abstract}
In the work \cite{AAG} of 1995, Alladi, Andrews, and Gordon provided a generalization of the two Capparelli identities involving certain classes of integer partitions. Inspired by that contribution, in particular as regards the general setting and the tools the authors employed, we obtain new partition identities by identifying further sets of partitions that can be explicitly put into a one-to-one correspondence by the method described in the 1995 paper. As a further result, although of a different nature, we obtain an analytical identity of Rogers-Ramanujan type, involving generating functions, for a class of partition identities already found in that paper and that generalize the first Capparelli identity and include it as a particular case. To achieve this, we apply the same strategy as Kanade and Russel did in a recent paper. This method relies on the use of jagged partitions that can be seen as a more general kind of integer partitions.
\end{abstract}

\section{Introduction}

In a 1969 paper, \cite{GEA1}, Andrews characterized the type of partition sets that could be set into a bijection using Euler's classical trick to show that partitions of $n$ into distinct parts are as many as partitions of $n$ into odd parts. In particular,
Andrews proved that an identity of Schur (\cite{SCH}) and one of G\"ollnitz (\cite{GOL}) provide examples of ``Euler-pairs''. Inspired by that paper, here we look at one of the identities given in \cite{CAP1}, see also \cite{CAP2}. We study the bijections provided by Alladi, Andrews, and Gordon in \cite{AAG} and we find new sets of partitions that can be set into a bijection using the same approach. For further details on this subject we refer the reader  to \cite{GEA, GEA2,CAP4,CAP3}.

As the starting point of our research we consider the partition identity which was proved in \cite{AAG} and in particular the Concluding Remarks (Section 7), according to which the first Capparelli's identity (see \cite{CAP3}) can be generalized from modulo $3$ to modulo $t$ by means of suitable dilations. In Section~\ref{analytical} we find an analytical identity for the partition identity modulo $t$. This is done using the same method as in \cite{KR} to compute the generating functions of the sum side. In Section~\ref{sdense} we again look back at \cite{AAG}, this time by generalizing the machinery which led the authors to build up the partition identity. In particular, we study a different class of partition identities which are indexed by two coprime integers $s$ and $t$. As better clarified in Section~\ref{sdense}, these two parameters play different roles  and actually generalize the roles of $s=2$ and $t=3$ in \cite{AAG}. In the present setting we obtain a new family of partition identities.

\section{An analytical identity for some Alladi, Andrews, Gordon bijections}\label{analytical}

Let $n$, $t$ be positive integers with $t>2$ and denote by ${\cal C}(n)_t$ the set of partitions of $n$ with distinct parts that are either divisible by $t$ or congruent to $t\pm1\pmod{2t}$. Furthermore, let ${\cal D}(n)_t$ denote the set of partitions of $n$ with distinct parts larger than $1$ that are either divisible by $t$ or congruent to $\pm1$ $\pmod{t}$ and whose difference is at least $t+1$, with the following exception: the difference between two adjacent parts can be smaller than $t+1$ if they are both divisible by $t$ or their sum is divisible by $2t$. Alladi, Andrews, and Gordon proved the following proposition (see \cite{AAG}).

\begin{proposition} \label{prop2xt1}
${\cal C}(n)_t$ e ${\cal D}(n)_t$ have the same cardinality.
\end{proposition}

In the present section we find an analytical expression for the above partition identity. The proof of our result is inspired by the argument in \cite{KR}. In accordance with that paper, we provide the following definition.
\begin{definition}
For a fixed positive integer $k$, a \emph{$k$-jagged partition} is a finite sequence $(a_1,\ldots,a_m)$ such that $a_1,\ldots,a_m\in\Z$, $a_1>0$, and $a_{i+1}- a_i\ge -k$, for every $i=1,\ldots,m-1$.
\end{definition}
Clearly, if $k=0$, we obtain the classical partitions written in weakly increasing order, as in \cite{KR}. The set of classical partitions $(b_1,\ldots,b_m)$ is easily seen to be in bijection with the set of $k$-jagged partitions $(a_1,\ldots,a_m)$ by associating $b_i$ to $a_i+(i-1)k$; essentially, we add a \emph{$k$-staircase}. In the present paper, this bijection plays an important role when passing to generating functions. For our purposes, we now consider a special class of $k$-jagged partitions.
\begin{definition}
A $k$-jagged partition is called \emph{strong} if it satisfies the condition $a_{i'}- a_i\ge -k$, for every $i=1,\ldots,m$ and $i'=i+1,\ldots,m$.
\end{definition}
With this definition it is now possible to introduce the key concept we require for our proof.
\begin{definition}
For each positive integer $j$ and strong $k$-jagged partition $\mu$, we define the \emph{maximal block} $\mathcal{M}_j$ corresponding to $j$ in the following way. If there is an element $a_i$ of $\mu$ such that $a_i=j$ and with the property that every element before $a_i$ is smaller than $j$, then $\mathcal{M}_j$ is the maximal subsequence of $\mu$ starting with $a_i$ whose elements belong to the set $\{j,j-1,j-2,\ldots,j-k\}$. If there is no element $a_i$ satisfying the above conditions, then $\mathcal{M}_j$ is the empty set.

\end{definition}
The maximal blocks are in bijection with the positive integers and it is not hard to see that a given strong $k$-jagged partition is exactly the juxtaposition of all its maximal blocks. While it seems difficult to work out the general form of a maximal block, in our context such blocks enjoy some additional properties which make their description easy in order to obtain Theorem \ref{analid}.
\begin{example}
We have that
\[
\mu=(3,5,5,4,5,6,4,3,4,0,-2,5,11)
\]
is a $4$-jagged partition and it is not a strong $4$-jagged partition. However it is a $8$-jagged partition and a strong $8$-jagged partition. If we regard $\mu$ as a strong $8$-jagged partition, we have the following maximal blocks:
\begin{align*}
&\mathcal{M}_3=(3), \\
&\mathcal{M}_5=(5,5,4,5), \\
&\mathcal{M}_6=(6,4,3,4,0,-2,5), \\
&\mathcal{M}_{11}=(11), \\
&\mathcal{M}_j=\varnothing, \quad \text{for } j=1,2,4,7,8,9,10 \text{ and } j\ge 12.
\end{align*}
Adding the $8$-staircase $(0,8,16,24,32,40,48,56,64,72,80,88,96)$ we get the classical partition $(3, 13, 21, 28, 37, 46, 52, 59, 68, 72, 78, 93, 107)$.
\end{example}
In the proof of the following theorem we need some further notation. Let $b$ be a finite subsequence; by $b^*$ we mean a string of either 0 or more contiguous blocks of $b$, and by $b^\bullet$ we mean either the empty string or  $b$ itself. For instance, the notation $(6,4,6,3)^*(6,5)^\bullet$ is compatible with any of the following: $(6,5)$, $(6,4,6,3)$,  $(6,4,6,3,6,4,6,3,6,4,6,3)$, or  $(6,4,6,3,6,4,6,3,6,5)$. But not with $(6,4,6,3,6,5,6,5)$ or  $(6,4,6,3,6,5,6,4,6,3)$.
\begin{theorem}\label{analid}
Let $t$ be an integer greater than $3$ and
\[
Q_t(a,b,c,d)=2ta^2+\frac{t}{2}b^2+tc^2+td^2+2tab+2tac+2tad+tbc+tbd+tcd.
\]
We have
\[
\prod_{n\equiv 0,t-1,t,t+1\!\!\!\!\!\pmod {2t}}\!\!\!\!\!\!\!\!\!\!\!\!\!\!\!\!\!\!\!\!(1+q^n) = \sum_{a=0}^{\infty}\sum_{b=0}^{\infty}\sum_{c=0}^{\infty}\sum_{d=0}^{\infty} \frac{q^{Q_t(a,b,c,d)+\frac{t}{2}b-c+d}}{(q^{2t},q^{2t})_a(q^t,q^t)_b(q^t,q^t)_c(q^t,q^t)_d},
\]
where the left hand side is the generating function of the partitions with distinct parts congruent to $0,t-1,t,t+1\pmod {2t}$, and the right hand side is the generating function of partitions with distinct parts greater than $1$ and congruent to $0,\pm 1 \pmod t$ such that the difference between consecutive parts is at least $t+1$ unless they are both $0 \pmod t$ or their sum is $0 \pmod {2t}$.
\end{theorem}
\begin{proof}
It is straightforward to see that the infinite product is the generating function of $\mathcal{C}(n)_t$. With Proposition \ref{prop2xt1} in mind, it suffices to show that the quadruple sum is the generating function of $\mathcal{D}(n)_t$. For every positive integer $j$, the configurations that are not allowed in $\mathcal{D}(n)_t$ are the following:

\begin{itemize}
    \item $j,j$;
    \item $tj-1,tj$;
    \item $tj-1,tj+t-1$;
    \item $tj,tj+1$;
    \item $tj,tj+t-1$;
    \item $tj+1,tj+t-1$;
    \item $tj+1,tj+t$;
    \item $tj+1,tj+t+1$.
\end{itemize}
If we subtract a $t$-staircase, we get
\begin{itemize}
    \item $j,j-t$;
    \item $tj-1,tj-t$;
    \item $tj-1,tj-1$;
    \item $tj,tj-t+1$;
    \item $tj,tj-1$;
    \item $tj+1,tj-1$;
    \item $tj+1,tj$;
    \item $tj+1,tj+1$.
\end{itemize}

Hence, the maximal blocks are the following:
\begin{itemize}
    \item $\mathcal{M}_{tj-1}=(tj-1,tj-t+1)^*(tj-1)^{\bullet}$;
    \item $\mathcal{M}_{tj}=(tj)^*$;
    \item $\mathcal{M}_{tj+1}=(tj+1)^{\bullet}$.
\end{itemize}
It follows that the generating function is
\begin{align*}
&\prod_{j=1}^{\infty}\frac{1+xq^{tj-1}}{1-xq^{tj-1}\cdot xq^{tj-t+1}}\prod_{j=1}^{\infty}\frac{1}{1-xq^{tj}}\prod_{j=1}^{\infty}(1+xq^{tj+1})= \\
&=\frac{(-xq^{t-1},q^t)_{\infty}(-xq^{t+1},q^t)_{\infty}}{(x^2q^t,q^{2t})_{\infty}(xq^t,q^t)_{\infty}}= \\
&=\sum_{a=0}^{\infty}\frac{x^{2a}q^{ta}}{(q^{2t},q^{2t})_a}\sum_{b=0}^{\infty}\frac{x^{b}q^{tb}}{(q^t,q^t)_b}\sum_{c=0}^{\infty}\frac{x^{c}q^{(t-1)c}q^{\frac{tc(c-1)}{2}}}{(q^t,q^t)_c}\sum_{d=0}^{\infty}\frac{x^{d}q^{(t+1)d}q^{\frac{td(d-1)}{2}}}{(q^t,q^t)_d}= \\
&=\sum_{a=0}^{\infty}\sum_{b=0}^{\infty}\sum_{c=0}^{\infty}\sum_{d=0}^{\infty}\frac{x^{2a+b+c+d}q^{4a+4b+2c^2+c+2d^2+3d}}{(q^{2t},q^{2t})_a(q^t,q^t)_b(q^t,q^t)_c(q^t,q^t)_d}.
\end{align*}
Now we add a $t$-staircase which, in terms of generating functions, corresponds to the substitution $x^m\mapsto x^mq^{\frac{tm(m-1)}{2}}$. Moreover, since we are not interested in taking account of the number of parts, we set $x=1$ and this gives the thesis.
\end{proof}

\section{Partition identities for the ``$s$-rate\,,\,$t$-stack'' case}\label{sdense}

As mentioned in the Introduction, in the present section we construct an original class of partition identities which are indexed by two coprime integers $s$ and $t$. Given a partition of a positive integer $n$ into distinct parts, we list its parts in decreasing order, as in \cite{AAG}. Fix three positive integers $n, s, t$, with $s, t$ coprime and greater than 1; denote by ${\cal C}(n)_s^t$ the set of partitions of $n$ with distinct parts multiple of $s$ or multiple of $t$. Let $W=\{hs+kt\colon h,k\in \mathbb{N}\}$ and $U=\mathbb{N}-W$, where $\mathbb{N}$ is the set of nonnegative integers. Notice that $U$ is finite because its largest element is $(s-1)(t-1)-1$.\par
Now denote by ${\cal D}(n)_s^t$ the set of partitions of $n$ with distinct parts $d_1,\ldots, d_m$, where the elements $d_{i_1}>d_{i_2}>\cdots >d_{i_p}$ are precisely those not congruent to $0\pmod t$, and with the following conditions which all parts $d_i$ must fulfil.

\begin{description}
\item{D0.}   Setting
\begin{equation}\label{D0}
  \begin{split}
    &f_p=d_{i_p}-(m-i_p)t,\\
    &f_{p-1}=d_{i_{p-1}}-(m-i_{p-1}-1)t,\\
&\qquad \vdots\\
    &f_{p-h}=d_{i_{p-h}}-(m-i_{p-h}-h)t,\\
    &\qquad \vdots\\
        &f_{1}=d_{i_{1}}-(m-i_{1}-p+1)t,
  \end{split}
\end{equation}

we require that $f_p$ be congruent to $0$ or $t\pmod s$, and the same must hold for $f_i-f_{i+1}$, with $i=1,\ldots, p-1$.

\item{D1.} $d_i\in W$.

\item{D2.}  If $d_i\equiv 0\pmod t$, then $d_i>t(m-i)$.

\item{D3.}
If $d_i - d_{i+r}<t+1$ for some positive integer $r$, then at least one of the following conditions must be satisfied:
\begin{description}
\item{I.}  $d_i - d_{i+r}\not\equiv 0 \pmod s$  and $d_i\equiv d_{i+r}\equiv 0 \pmod t$;
\item{II.} $d_i - d_{i+r}=sj$ and $d_i + d_{i+r}\not\equiv\pm sj \pmod{st}$.
\end{description}

\end{description}

Notice that ${\cal C}(n)_2^t$ has a different meaning from ${\cal C}(n)_t$; a similar remark concerns ${\cal D}(n)_2^t$ and ${\cal D}(n)_t$. We are going to prove the following result.

\begin{proposition} \label{propsxts1}
${\cal C}(n)_s^t$ and ${\cal D}(n)_s^t$ have the same cardinality.
\end{proposition}

In order to establish the above proposition we prove a stronger result, namely, Proposition \ref{propsxts2}, for which some more terminology is needed. Denote by ${\cal C}(n;i_1,\ldots,i_{t-1},k)_s^t$ the subset of ${\cal C}(n)_s^t$ whose parts are grouped  according to their congruence class as follows:

\noi
$i_1$ parts congruent to $s \pmod{st}$, $\ldots, i_{t-1}$ parts congruent to $(t-1)s \pmod{st}$, $k$ parts larger than $t(\sum_{j=1}^{t-1}i_j)$ and congruent to $0 \pmod{t}$.
Similarly, set ${\cal D}(n;i_1,\ldots,i_{t-1},k)_s^t$ to be the subset of ${\cal D}(n)_s^t$ having $i_h$ parts congruent to $hs \pmod{t}$, with $1\le h\le t-1$, and $k$ further parts congruent to $0 \pmod{t}$.

\noi

Now we proceed with the proof of the stronger result.

\begin{proposition}\label{propsxts2}
${\cal C}(n;i_1,\ldots,i_{t-1},k)_s^t$ and ${\cal D}(n;i_1,\ldots,i_{t-1},k)_s^t$ have the same cardinality.
\end{proposition}

\begin{proof}

Given a partition in ${\cal C}(n;i_1,\ldots,i_{t-1},k)_s^t$, we associate to it a partition in ${\cal D}(n;i_1,\ldots,i_{t-1},k)_s^t$, with an algorithm that generalizes the classical case $t=3$ in \cite{AAG}; also our terminology traces back to that paper. Later we show that such a procedure is reversible.

Let $\pi\in{\cal C}(n;i_1,\ldots,i_{t-1},k)_s^t$.

\s\noi
{\bf Step 1)} Setting $p=\sum_{j=1}^{t-1}i_j$, split $\pi$ into the subpartition  $\pi_1=(a_1,\ldots,a_p)$ of those elements not congruent to $0 \pmod {t}$ and the subpartition $\pi_2$ made up of those  elements congruent to $0 \pmod t$.
Notice that two elements  in $\pi_1$ have difference  $sj\le t+1$ only if they have sum not congruent to $\pm sj\pmod{st}$, i.e., they satisfy D3-II.
Indeed, suppose we have two parts $s\alpha, s\beta \in \pi_1$ such that $s\alpha -s\beta=sj$ and $ sj\leq t+1$. Assuming, by contradiction, that $s\alpha+s\beta\equiv\pm sj \pmod{st}$, since $\alpha=\beta+j$ we have either $2s\beta+sj\equiv sj \pmod{st}$ or $2s\beta+sj \equiv -sj \pmod{st}$. In the first case $2\beta\equiv 0 \pmod t$, against the assumptions. In the second, likewise, we have $2\alpha\equiv 0 \pmod t$, once more contradicting the assumption.

\s\noi
{\bf Step 2)} Split $\pi_2$ into the subpartitions $\pi_5$ and $\pi_4$ consisting respectively of those elements larger than $tp$ and those not greater. Set $\pi_5=(b'_1,\ldots,b'_k)$ and $\pi_4=(b''_1,\ldots,b''_r)$.

\s\noi
{\bf Step 3)} We construct the  $t$-fold conjugate of $\pi_4$, in symbols $\pi_4^*$. If $\pi_4=\emptyset$, then $\pi_4^*=\emptyset$. Otherwise, set $b_1''=u_1t, b_2''=u_2t,\ldots, b_r''=u_rt$, with $u_1\geq u_2\geq \ldots \geq u_r$, then $\pi_4^*$ is the partition whose diagram  has $tr$ columns. Specifically,  consider  $u_r$ rows with cardinality $tr$, $u_{r-1}-u_{r}$ rows with cardinality $t(r-1)$,  $\ldots$ , $u_2-u_3$ rows with  cardinality $2t$, $u_1-u_2$ rows with  cardinality $t$. Notice that the columns, taken as blocks of  $t$ columns each, give the elements of $\pi_4$.

 For example, for $t=7$, if $\pi_4=(35,14,7)$, divide by 7 each part thus obtaining  $(5,2,1)$, then form a diagram by using ``blocks'' of  7 squares

\s
\ytableausetup{smalltableaux}
\ydiagram{7}

 \noi  by stacking, respectively,  5 blocks, 2 blocks, 1 block, as follows
  \s

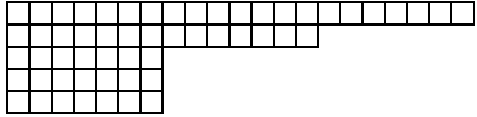
\ydiagram{21,14,7,7,7}

\s\noi thus getting  $\pi_4^*=(21,14,7,7,7)$.

\s\noi
{\bf Step 4)} Let $\alpha_1\ge \alpha_2\ge\cdots\ge \alpha_{u_1}$ where $\alpha_i$ is the cardinality of the $i$-th row of the diagram associated to $\pi_4^*$. Add the partitions $\pi_1$ and $\pi_4^*$ by adding the corresponding parts $a_i+\alpha_i$ for $1\le i\le u_1$ and leaving alone the elements $a_{u_1+1}\ldots a_p$. Notice that this is possible as  $u_1\leq p$, since $b_1''=u_1t\leq tp$.  We thus get a new partition $\pi_6=(a'_1,\ldots,a'_p)$, with $a'_i>a'_{i+1}$, for $i=1,\ldots,p-1$. This operation either leaves the differences between the parts unchanged or increases them by multiples of $t$ so that condition D3-II still holds, as well as condition D0 for the symbols $a'_i$ in place of $f_i$.

\s\noi
{\bf Step 5)} Construct a string $\pi_5/\pi_6$ by juxtaposing, left to right, first the elements of  $\pi_5$, then those of $\pi_6$. In this string $i_h$ parts are congruent to $hs \pmod{t}$, for $h=1,\ldots,t-1$, and $k$ parts congruent to $0 \pmod{t}$, the latter being larger than  $tp$.

\s\noi
{\bf Step 6)} Subtract multiples of $t$ to the elements of $\pi_5/\pi_6$, by obtaining the following new elements:
\begin{align*}
\bar b'_1&=b'_1-(p+k-1)t, \\
\bar b'_2&=b'_2-(p+k-2)t, \\
&\vdots \\
\bar b'_k&=b'_k-pt, \\
\bar a'_1&=a'_1-(p-1)t, \\
\bar a'_2&=a'_2-(p-2)t, \\
&\vdots \\
\bar a'_p&=a'_p-0t.
\end{align*}
Notice that while the elements $\bar b'_i$ remain in a nonincreasing order, this does not necessarily happen for the elements  $\bar a'_i$. To be more precise, the nondecreasing order fails whenever  $a'_i-a'_{i+1}=j<t$. In such a case $\bar a'_{i+1}=\bar a'_i+t-j>\bar a'_{i}$. Moreover the elements  $\bar b'_i$ are all strictly positive since they are above the threshold value $tp$, while the elements  $\bar a'_i$ may be negative.

\s\noi
{\bf Step 7)} Starting from the string  $S_0=(\bar b'_1,\ldots,\bar b'_k,\bar a'_1,\ldots,\bar a'_p)$, we define a recursive algorithm which will lead us to a final string  $S_f$ in $k$ steps. Define the generic $i$-th step, $1\leq i\leq k$. Place  $\bar b'_i$ in the string $S_{i-1}$ in the rightmost possible position so that all the elements to its left are larger than itself. At the end of the process we get the desired string $S_f$.

\s\noi
{\bf Step 8)} Denote by  $c_1,\ldots,c_{p+k}$ the elements of $S_f$. We construct the following elements $d_i$, for $i=1,\ldots, p+k$:
\begin{align*}
&d_1=c_1+(p+k-1)t, \\
&d_2=c_2+(p+k-2)t, \\
&\ \vdots \\
&d_{p+k-1}=c_{p+k-1}+t, \\
&d_{p+k}=c_{p+k}.
\end{align*}
The difference $\Delta$ between  two elements $d_i$ and $d_{i+1}$ that are not congruent to $0\pmod t$ may be less than $t+1$ only if $\Delta$ is a multiple of $s$ and the sum is not congruent to $\pm \Delta \pmod{t}$.

Instead, if $d_i\not\equiv0 \pmod t$ and $d_{i+1}\equiv0 \pmod t$ (or vice versa), we have $d_i=c_i+(p+k-i)t$ and $d_{i+1}=c_{i+1}+(p+k-i-1)t$. Since $c_{i+1} -c_i\ge1$, we deduce that $d_{i+1} -d_i\ge t+1$.

Notice that D0 holds. Indeed, the elements $f_i$ obtained from the elements $d_i$ not congruent to $0\pmod t$, by the formulas \eqref{D0}, coincide with $a'_i$ for which we already observed that they satisfy the requirements in Step 4. Therefore this algorithm transforms a partition $\pi$ in  ${\cal C}(n;i_1,\ldots,i_{t-1},k)_s^t$ into a partition in ${\cal D}(n;i_1,\ldots,i_{t-1},k)_s^t$ that we denote by $\pi_3$ again in accordance with \cite{AAG}.

Finally, we only need to show that this procedure is completely reversible. Given a partition $\tilde\pi$ of $n$ in ${\cal D}(n;i_1,\ldots,i_{t-1},k)_s^t$ it is trivial to trace back the steps up to Step 3. We thus get a  partition of $n$ made of  some elements not congruent to $0\pmod t$ and of some elements congruent to $0 \pmod{t}$.
The set of the first type of elements, consistently with previous notation, we denote by $\pi_6=(a'_1,\ldots,a'_p)$, with $a'_i>a'_{i+1}$, for $i=1,\ldots,p-1$.  The set of the second type of elements we denote by $\pi_5=(b'_1,\ldots,b'_k)$, with $b'_i>b'_{i+1}$, for $i=1,\ldots,k-1$.

Now, check whether  $a_p'$ is congruent to 0 or $t$ modulo $s$. In the first case do nothing. In the second case create a diagram with one row of  $t$ squares. Proceed inductively by creating, in corresponding with each $a_i'$, a row of squares to be stacked on top of the row corresponding to  $a_{i+1}'$, with the same number of squares of the row corresponding to  $a_{i+1}'$ if $a_i'\equiv a_{i+1}' \pmod s$, otherwise add  $t$ new squares to the row. We thus form a  diagram that gives, by using  stacks of $t$ columns, the elements of  $\pi_4$ and, using the rows (corresponding to $\pi_4^*$) the  quantities to be subtracted from the elements of $\pi_6$, in order to obtain  the elements of $\pi_1$. In this fashion, the partition $\pi$ in ${\cal C}(n,i_1,\ldots,i_{t-1},k)_s^t$ is completely reconstructed.
\end{proof}

We conclude this section with some remarks on the particular case $s=2$.

\begin{remark}
If $s=2$, the condition D0 is trivially satisfied.
\end{remark}

\begin{proposition}
Condition D2 is redundant if $s=2$ and $t=3$.
\end{proposition}

\begin{proof}
By contradiction, assuming that D2 does not hold, let $i_0$ be the largest integer such that $3\mid d_{i_0}\le3(m-i_0)$ and let $d_{i_0+h}$ be the next multiple of $3$ from left to right --- if there is no such multiple, set $h=m-i_0+1$ and define $d_{m+1}=0$. Clearly, $d_{i_0+h}\ge 3(m-i_0-h)+3$. Notice that $h$ must be larger than~$1$. The $h-1$ parts between $d_{i_0}$ and $d_{i_0+h}$ are all congruent to $1$ or $2 \pmod 3$. Now condition
D3 forces the leftmost part to be not larger than $d_{i_0}-4$ and the rightmost to be not smaller than $d_{i_0+h}+4$. Furthermore, using condition D3, it is easy to see that every interval of the form $[\alpha,\alpha+5]$ contains at most $2$ this $h-1$ parts. It follows that there are no more than
$$2\left\lceil\frac{d_{i_0}-4-(d_{i_0+h}+4)+1}{6}\right\rceil$$
parts of this kind. By hypotheses, if $h\le m-i_0$ such a number does not exceed $2\left\lceil\frac{3h-10}{6}\right\rceil$, which is less than $h-1$, a contradiction. In the remaining case, namely if $h=m-i_0+1$, the element $d_m$ might be equal to $2$ but the above argument is still valid, using similar calculations, as long as $h$ is even. Instead, if $h$ is odd (necessarily $h\ge3$), we slightly improve the argument. Since $d_{m-1}\ge4$, counting the $h-2$ parts from $d_{i_0+1}$ to $d_{m-1}$ leads to the following contradiction:
$$2\left\lceil\frac{d_{i_0}-4-4+1}{6}\right\rceil\le2\left\lceil\frac{3(h-1)-7}{6}\right\rceil\le h-3\ .$$
\end{proof}

Notice that for $t$ odd and larger than 3 Condition D2 is necessary.

\section{An example} \label{es 1}

We illustrate an example of the bijection when  $s=2$ and $t=7$. Let $\pi=(84,70,66,46,40,38,35,14,10,8,7,4,2)\in{\cal C}(424)_2^7$.

\m\noi
{\bf Step 1)} Split $\pi$ in
\[
\pi_1=(66,46,40,38,10,8,4,2),\quad \pi_2=(84,70, 35, 14, 7).
\]
So $p=8$ and the threshold is $tp=56$.

\m\noi
{\bf Step 2)} Split $\pi_2$ in
\[
\pi_5=(84,70), \quad \pi_4=(35, 14, 7),
\]
hence $k=2$ and $r=3$.

\m\noi
{\bf Step 3)} $\pi_4=( 35, 14, 7)\mapsto (5,2,1)$

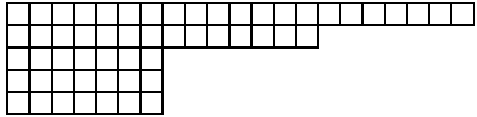
\s
\ytableausetup{smalltableaux}
\ydiagram{21,14,7,7,7}
\m\noi

\noindent which gives the partition $\pi_4^*=( 21, 14, 7, 7 ,7)$.

\m\noi
{\bf Step 4)} $\pi_6=\pi_1+\pi_4^*=(66+21,46+14,40+7,38+7,10+7,8,4,2)=(87,60,47,45,17,8,4,2)$.

\s\noi
{\bf Step 5)} Form the string $\pi_5/\pi_6=(84,70,87,60,47,45,17,8,4,2)$.

\s\noi
{\bf Step 6)} Subtract to the string just obtained multiples of  7 as follows
\begin{equation*}
  \setcounter{MaxMatrixCols}{12}
\begin{matrix}
    84&70&87&60&47&45&17&8&4&2&&-\\
    63&56&49&42&35&28&21&14&7&0&&=\\ \\
    21&14&38&18&12&17&-4&-6&-3&2&&
  \end{matrix}
\end{equation*}

\s\noi
{\bf Step 7)} Starting from the string $S_0$ obtained in the previous step, move 14 as far right as possible thus obtaining
$$
S_1=(21,38,18,14,12,17,-4,-6,-3,2).
$$
Next, move similarly the number  21 obtaining
$$
S_2=(38,21,18,14,12,17,-4,-6,-3,2)=S_f.
$$

\s\noi
{\bf Step 8)} Now we add again to  $S_f$ the string of multiples of 7 as before
\begin{equation*}
 \setcounter{MaxMatrixCols}{12}
\begin{matrix}
38&21&18&14&12&17&-4&-6&-3&2&&+ \\
63&56&49&42&35&28&21&14&7&0&&= \\ \\
101&77&67&56&47&45&17&8&4&2&&
\end{matrix}
\end{equation*}

\noi getting the  partition $\pi_3=(101,77,67,56,47,45,17,8,4,2)\in{\cal D}(424)_2^7$ as desired.

We now check that this process is reversible. Given the partition $\pi_3=(101,77,67,56,47,45,17,8,4,2)$, the steps 8,7,6,5 are easily   reversible and lead to $\pi_5$ and $\pi_6=(84,70,87,60,47,45,17,8,4,2)$. To recover  $\pi_4$ proceed as follows. Note the position where, starting from the right, we find the first odd part.
In our example it is 17 and, correspondingly, we draw a row of 7 squares. Going leftward and ignoring the multiple of 7, we have the sequence $(17,45,47,60,87)$ that is congruent to $(1,1,1,0,1)$ modulo 2. For each element with the same parity of the previous one we draw a row upon the others with the same number of squares than the row below it. For each element with a different parity than the previous one, we draw a row upon the others with 7 squares more than the row below it. This generates the diagram

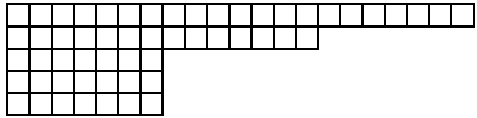
\s
\ytableausetup{smalltableaux}
\ydiagram{21,14,7,7,7}
\m\noi

\noindent that coincides with the diagram in step 3. In this diagram there are 3 blocks of 7 columns containing 35, 14 and 7 squares respectively. Hence $\pi_4=(35,14,7)$. The rows of the diagram give the string $(21,14,7,7,7)$ which must be subtracted to $(87,60,47,45,17)$ to get $\pi_1$.\par

\vskip1.5cm

\noindent{\bf \Large Acknowledgements}\par
\bigskip
\noindent This manuscript was prepared with the funding support of \emph{Progetti di Ateneo, Sapienza Universit\`a di Roma} and of the research grant ``Ing. Giorgio Schirillo'' of the Istituto Nazionale di Alta Matematica ``F. Severi'', Rome.

\end{document}